\documentclass[12pt]{amsart}
\usepackage{amsthm,amsfonts,fullpage}
\usepackage{color}

\input{xy}
\xyoption{all}
\xyoption{matrix}

\theoremstyle{plain}
\newtheorem{thm}{Theorem}
\newtheorem{teo}{Theorem}
  \newtheorem{lem}[thm]{Lemma}
  \theoremstyle{remark}
  \newtheorem{rem}[thm]{Remark}
 
\theoremstyle{definition}
  \newtheorem{example}[thm]{Example}
 \theoremstyle{definition}
  
  \theoremstyle{plain}
  \newtheorem{prop}[thm]{Proposition}
  \theoremstyle{plain}
  \newtheorem{cor}[thm]{Corollary}

\def\t{\triangleleft}

\def\sub{\underline}
\def\ot{\otimes}
\def\wh{\widehat}
\def\wt{\widetilde}
\def\Hom{\mathrm{Hom}}
\def\ignore#1{}

\def\Id{\mathrm{Id}}
\newcommand{\mop}[1]{\mathop{\hbox {\rm #1} }\nolimits}
\title{Bialgebraic proof of the existence of cup product in the cohomology of
racks and quandles}
\author{Marco Farinati}
\address{Dto. de Matem\'atica FCEyN, Universidad de Buenos Aires - IMAS Conicet}
\author{Dominique Manchon}
\address{LMBP - UMR 6620,
	CNRS-Universit\'e Clermont-Auvergne,
	3 place Vasar\'ely, CS 60026,
        63178 Aubi\`ere, France.}
        \email{manchon@math.univ-bpclermont.fr}
        \urladdr{http://recherche.math.univ-bpclermont.fr/~manchon/}
\author{Simon Covez}
\address{Laboratoire de Math\'ematiques Jean Leray, Universit\'e de Nantes, 2 chemin de la Houssini\`ere, 44322 Nantes}
\email{simon.covez@gmail.com}
\begin{document}\maketitle
\begin{abstract}
We retrieve the graded commutative algebra structure of rack and quandle cohomology by purely algebraic means.
\end{abstract}

\section*{Introduction}

The purpose of this note is to point out to a natural differential graded (associative) bialgebra attached to
any rack $X$, that governs algebraic structures on the rack homology and cohomology of $X$. More precisely,
we prove that this bialgebra induces a structure of differential graded commutative algebra on the complex
computing rack cohomology. Cup product on rack cohomology was discovered and described in topological terms by
Clauwens \cite{Cl}: we prove that our purely algebraic construction gives back
Clauwens' cup product. \\

\noindent\textbf{Acknowledgements:} M. F. is Research member of Conicet, partially supported by 
PIP 112-200801 -00900,
PICT 2006 00836, UBACyT X051
and MathAmSud 10-math-01 OPECSHA.

\section{Basic definitions}
A rack is a nonempty set $X$ together with a binary operation $\t:X\times X\to X$: 
$(x,y)\mapsto x\t y$ (sometimes denoted also by $x\t y=x^y$) satisfying the following two axioms:
\begin{itemize}
\item[R1] $-\t y:X\to X$ is a bijection for all $y\in X$, and
\item[R2] $(x\t y)\t z=(x\t z)\t (y\t z)$, for all $x,y,z\in X$.
\end{itemize}
In exponential notation, the second axiom reads 
$(x^y)^z=\left(x ^z\right)
^{\left(y ^z\right)}$. \\

The first family of examples is to take $X$ a group, and $x\t y=y ^{-1}xy$. A rack satisfying
$x\t x=x$ for all $x\in X$, as this example does, is called a Quandle. We refer
\cite{AG} and references therein for several examples and a brief history of racks and quandles.\\

Let $k$ be a commutative ring with $1$, and define $C_n(X,k)=k[X^n]$= the free $k$-module with basis $X^n$, and
$C ^n(X,k)=k^{X  ^n}\cong \Hom(C_n(X,k),k)$ with differentials $\partial:C_n(X,k)\to C_{n-1}(X,k)$ defined by
\begin{equation}\label{bord-racks}
\partial(x_1,\cdots,x_n)=
\sum_{i=1}^{n}(-1)^i(x_1,\cdots,\wh{x_i}, \cdots ,x_n)
-\sum_{i=1}^{n}(-1)^i(x_1^{x_i},\cdots,x_{i-1}^{x_i},\wh{x_i}, \cdots ,x_n)
\end{equation}
where $\wh x_i$ means that this element was omitted. For cohomology,
the differential is $\partial ^*:C^n\to C ^{n+1}$.
These maps are of square zero (by direct computation, or see remark \ref{rem2} later) and define respectively the homology and cohomology
of the rack $X$ with coefficients in $k$. The introduction of the algebraic objects in the next section is aimed to provide a simple proof that $(C  ^\bullet(X,k),\partial^*)$ is a differential graded algebra.
\section{Algebra and d.g. bialgebra associated to a rack}
Fix a rack $X$ and a commutative ring $k$ with unit. 
Let $A(X)$ (denoted simply by $A$ if $X$ is understood) the quotient of the free $k$-algebra on generators $X$
modulo the ideal generated by elements of the form $yx ^y-xy$:
\[
A=k\langle X\rangle/\langle yx^y-xy\rangle
\]
It can be easily seen that
$A$ is a $k$-bialgebra declaring $x$ to be grouplike for any $x\in X$, since $A$ agrees with the semigroup algebra on the monoid
freely generated by $X$ with relations $yx^y\sim xy$.
If one considers the group $G_X$ freely generated by
$X$ with relations $yx=xy^x$, then $k[G_X]$ is the (non commutative) localization of $A$, where one has inverted the
elements of $X$.
An example of $A$-bimodule that will be used later, which is actually a $k[G_X]$-module, is
 $k$ with $A$-action determined on generators by
\[
x\lambda y=\lambda, \ \forall x,y\in X
\]
We define $B(X)$ (also denoted by $B$) as the algebra freely generated by two copies of $X$, with the following relations:
\[
B=k\langle x, e_y: x,y\in X\rangle/\langle yx^y-xy,ye_{x^y}-e_xy\rangle
\]
The key result is the following:
\begin{teo}\label{dgb}
 $B$ is a differential graded bialgebra.
\end{teo}
By differential graded bialgebra we mean that the differential is both a derivation with respect to multiplication, and 
coderivation with respect to comultiplication.
\begin{proof}

The grading is given by declaring $|e_x|=1$ and $|x|=0$ for all $x\in X$. Since the relations are homogeneous, $B$ is a graded algebra.
Moreover, define $d:B_\bullet\to B_{\bullet-1}$ the unique (super) derivation of degree -1 determined by
\[
d(e_x)=1-x,\ d(x)=0,\ \forall x\in X
\]
In order to see that $d$ is well-defined, one must check that the relations
$yx^y\sim xy$ and $ye_{x^y}\sim e_xy$ are compatible with $d$. The first relation
is easier since
\[
d(yx^y-xy)=
d(y)x^y+yd(x^y)-d(x)y-xd(y)=0+0-0-0=0
\]
For the second relation:
\[
d(ye_{x^y}-e_xy)=
yd(e_{x^y})-d(e_x)y=
y(1-x^y)-(1-x)y=
y-yx^y-y+xy=xy-yx^y
\]
So we see that the ideal of relations defining $B$ is stable by $d$.
It is clear that $d ^2=0$, since $d^2$ is zero on generators, hence we get a structure of differential graded algebra on $B$. The comultiplication in $B$ is defined by
\[
\begin{array}{rrcl}
\Delta:&B&\longrightarrow&B\ot B\\
&x&\mapsto& x\ot x\\
&e_x&\mapsto& e_x\ot x+1\ot e_x
\end{array}\]
and  extended multiplicatively, using the standard $k$-algebra structure
on the tensor product (in the graded sense, with Koszul signs taken into account). Notice that $B$ is not cocommutative. We need to check first that $\Delta$ is well defined. The first relation is the easiest since
elements of $X$ are group-like:
\[
\begin{array}{ccl}
\Delta(xy)-\Delta(yx^y)
&=&
(x\ot x)(y\ot y)-
(y\ot y)(x^y\ot x ^y)
\\
=
xy\ot xy-yx^y\ot yx^y
&=&
xy\ot (xy-yx^y)+(xy-yx^y)\ot yx^y.
\end{array}
\]
For the second relation we check:
\[
\begin{array}{rcl}
\Delta(ye_{x^y}-e_xy)&
=&
(y\ot y)(e_{x^y}\ot x^y+1\ot e_{x^y})
-(e_x\ot x+1\ot e_x)(y\ot y)\\
&=&
ye_{x^y}\ot yx^y+y\ot ye_{x^y}
-e_xy\ot xy-y\ot e_xy\\
&=&
(ye_{x^y}-e_xy)\ot yx^y
+
e_xy\ot (yx^y-xy)
+y\ot (ye_{x^y}- e_xy),
\end{array}
\]
so we see that the ideal defining the relations is also a coideal. One checks then that $d$ is a (super) co-derivation: in order to do that, it is enough
to check that on generators:
\begin{eqnarray*}
\Delta(d(x))&=&\Delta(0)=0=d(x)\ot x+x\ot d(x)=(d\ot 1+1\ot d)(\Delta x),\\
\Delta(d(e_x))&=&\Delta(1-x)=1\ot 1-x\ot x.
\end{eqnarray*}
On the other hand,
\begin{eqnarray*}
(d\ot 1+1\ot d)(\Delta e_x)&=&
(d\ot 1+1\ot d)(e_x\ot x+1\ot e_x)\\
&=&
(1-x)\ot x+1\ot (1-x)=
1\ot x-x\ot x+1\ot 1-1\ot x\\
&=&
1\ot 1-x\ot x.
\end{eqnarray*}
This ends up the proof of Theorem \ref{dgb}.
\end{proof}

 \begin{example}
\label{exDelta}
$\Delta(e_xe_y)=e_xe_y\ot xy+1\ot e_xe_y+
e_x\ot xe_y
-e_y\ot ye_{x^y}$ because
$\Delta(e_xe_y)$ is by definition $\Delta(e_x)\Delta(e_y)$ in $B\ot B$, and this is equal to
\begin{eqnarray*}
(e_x\ot x+1\ot e_x)(e_y\ot y+1\ot e_y)
&=&e_xe_y\ot xy+ e_x\ot xe_y
-e_y\ot e_xy+1\ot e_xe_y\\
&=&e_xe_y\ot xy+ e_x\ot xe_y
-e_y\ot ye_{x^y}+1\ot e_xe_y.
\end{eqnarray*}
\end{example}
\noindent The relation between $B$ and the homology and cohomology is given in the following:
\begin{lem}\label{Bvscomplejo}
There is an isomorphism of left $A$-modules $B\cong A\ot_kTX$ where $TX$ is the free unital algebra generated by $X$. It induces an isomorphism of complexes
\[
 (C_\bullet,\partial)\cong (k\ot_AB_\bullet,\Id_k\ot_Ad)\hbox{, and }(C^\bullet,\partial^*)\cong (\Hom_{A-}(B,k),d^*)
\]
 where
$\Hom_{A-}$ means that the left $A$-structure is used to compute Hom.
\end{lem}
\begin{proof}
It is clear from the relation $e_xy=ye_{x^y}$, that any noncommutative monomial
in the variables $y_1,\dots,y_k,e_{x_1},\dots,e_{x_n}$ may be written in the form
$y'_1y'_2\cdots y_k'e_{x'_1}\cdots
e_{x'_n}$,
for instance,  $xe_yze_t=xze_{y^z}e_t$. We may identify $TX$ with the subalgebra of $B$
generated by $\{e_x:x\in X\}$ and $A$ with the subalgebra of $B$ generated by
$\{x:x\in X\}$ and, written in ``canonical form'' we get the isomorphism
of left $A$-modules $B\cong A\ot TX$. This implies
$k\ot_AB\cong TX$ as $k$-modules, the isomorphism being
\[
1\ot e_{x_1}\cdots e_{x_n}\mapsto (x_1,\dots,x_n)
\]
In order to compute the differential, we use that $d$ is a superderivation:
\[
(1\ot d)(
1\ot e_{x_1}\cdots e_{x_n})
=\sum_{i=1}^ n(-1)^ {i+1}
1\ot e_{x_1}\cdots e_{x_{i-1}}d(e_{x_i})e_{x_{i+1}}\cdots e_{x_i}
\]
\[
=\sum_{i=1}^ n(-1)^ {i+1}
1\ot e_{x_1}\cdots e_{x_{i-1}}(1-x_i)e_{x_{i+1}}\cdots e_{x_i}
\]
\[
=\sum_{i=1}^ n(-1)^ {i+1}
1\ot e_{x_1}\cdots e_{x_{i-1}}e_{x_{i+1}}\cdots e_{x_i}
-
\sum_{i=1}^ n(-1)^ {i+1}1\ot e_{x_1}\cdots e_{x_{i-1}}x_ie_{x_{i+1}}\cdots e_{x_i}.
\]
Using the relation $e_xy=ye_x^ y$ and the triviality of the action  one gets
\[
1\ot e_{x_1}\cdots e_{x_{i-1}}x_ie_{x_{i+1}}\cdots e_{x_i}
=1\ot  x_i e_{x_1^{x_i}}\cdots e_{x_{i-1}^{x_i}}e_{x_{i+1}}\cdots e_{x_i}
\]
\[
=1 \cdot x_i\ot  e_{x_1^{x_i}}\cdots e_{x_{i-1}^{x_i}}e_{x_{i+1}}\cdots e_{x_i}
=1 \ot  e_{x_1^{x_i}}\cdots e_{x_{i-1}^{x_i}}e_{x_{i+1}}\cdots e_{x_i}
\]
that maps to
$(x_1^{x_i},\cdots,x_{i-1}^{x_i},\wh{x_i}, \cdots ,x_n)$, in agreement
with the definition of $\partial$.
\end{proof}

\begin{rem}\label{rem2}
 One may use this lemma to provide a  very simple proof that $\partial ^ 2=0$ in 
$C_\bullet(X,k)$.
\end{rem}

\noindent As a corollary, one gets the main result:

\begin{thm}\label{thm:main}
The complex $(C^\bullet(X,k),\partial^*)$ admits a strictly associative product and
$ \partial^*$ is a (super) derivation with respect to it.
\end{thm}
\begin{proof}
Since the elements of $A\subset B$ are group-like, if we consider the $A$-diagonal
structure on $B\ot B$ (i.e. $x\cdot (b\ot b')=xb\ot xb'$ for all $x\in X$)
then $\Delta:B\to B\ot B$ is a morphism of $A$-modules. Let us denote $B\ot ^DB$
the $A$-module $B\ot B$ with this diagonal action. Using the $A$-module $k$, one can dualize
 the map $\Delta$ and get:
\[
\Delta^*:\Hom_{A-}(B\ot^DB,k)\to \Hom_{A-}(B,k)
\]
One also has the natural map $i:\Hom(B,k)\ot \Hom(B,k)\to \Hom(B\ot B,k)$ defined on homogeneous elements by $i(f\ot g)(b\ot b')=(-1)^{|g||b|}f(b)g(b')$. We claim
that if we consider the restriction $i|$ of $i$ to 
$\Hom_{A-}(B,k)\ot
\Hom_{A-}(B,k)$, then the image of $i|$ is contained in
$\Hom_{A-}(B\ot^DB,k)
\subset \Hom(B\ot B,k)$, namely, that we have a commutative diagram
\[
\xymatrix{ 
\Hom(B,k)\ot\Hom(B,k)\ar[r]^i&\Hom(B\ot^DB,k)\\
\Hom_{A-}(B,k)\ot\Hom_{A-}(B,k)\ar@{^(->}[u]\ar@{-->}[r]^-{i|}&\Hom_{A-}(B\ot^DB,k)\ar@{^(->}[u]
}
\]
In order to prove the claim, we consider $A$-linear maps  $f$ and $g$
from $B$ to $k$ (recall the action on $k$ is trivial, i.e. 
$x\lambda=\lambda$ for all $x\in X$, $\lambda\in k$).
Let us compute:
\[
i(f\ot g)\big(x\cdot (b\ot b')\big)
=i(f\ot g)(xb\ot xb')
=(-1)^{|b||g|}f(xb) g(xb')\]
\[
=(-1)^{|b||g|}\big(xf(b)\big)\big(x g(b')\big)
=(-1)^{|b||g|}f(b) g(b')
=(-1)^{|b||g|}x\big((f(b) g(b')\big)
=x\big(i(f\ot g)(b\ot b')\big).
\]
As a consequence, one can compose $i|$ with $\Delta^*$, and in this way we
define the multiplication
\begin{equation}\label{cup}
\smile:\Hom_{A-}(B,k)\ot \Hom_{A-}(B,k)\longrightarrow Hom_{A-}(B,k).
\end{equation}
This product $\smile$ is associative because 
$\Delta$ is coassociative,  $(B\ot ^DB)\ot ^DB=B\ot^D(B\ot^DB)$
and $i|$ is compatible with this equality.
Finally $\partial^*$ is a derivation because it identifies with $d^*$,
and $d$ is a coderivation with respect to $\Delta$ in $B$.
\end{proof}
\begin{example}
Let $f,g\in C^2(X,k)$, in order to compute $f\smile g$ one needs to compute
the summands in $\Delta(e_xe_ye_ze_t)$ with two tensors of type $e_{x_i}$ on each factor:
\[
\Delta(e_xe_ye_ze_t)=
\Delta(e_xe_y)\Delta(e_ze_t)=\]
\[=
(e_xe_y\ot xy+1\ot e_xe_y+
e_x\ot xe_y
+e_y\ot ye_{x^y})
(e_ze_t\ot zt+1\ot e_ze_t+
e_z\ot ze_t
+e_t\ot te_{z^t})
\]
\[
=e_xe_y\ot xye_ze_t+e_ze_t\ot e_xe_yzt
+e_x e_z\ot   xe_yze_t
+e_xe_t   \ot xe_yte_{z^t}
+e_y e_z\ot ye_{x^y}ze_t
+e_y  e_t\ot ye_{x^y}te_{z^t}+\cdots
\]
where the dots are terms in which  $f$ and $g$ vanish. Reordering the $e_{x_i}$'s on the right and the elements
of $A$ on the left we get
\[
=e_xe_y\ot xye_ze_t+e_ze_t\ot zx e_x^{zt}e_y  ^{zt}
+e_x e_z\ot   xze_{y^z}e_t
+e_xe_t   \ot xte_{y^t}e_{z^t}
+e_y e_z\ot yze_{x^{yz}}e_t
+e_y  e_t\ot yte_{x^{yt}}e_{z^t}+\cdots
\]
so finally 
$(f\smile g)(e_xe_ye_ze_t)$
is equal to
\[
f(e_xe_y)g(e_ze_t)+f(e_ze_t)g( e_x^{zt}e_y  ^{zt})
+f(e_x e_z)g(e_{y^z}e_t)\]
\[
+f(e_xe_t)g(e_{y^t}e_{z^t})
+f(e_y e_z)g(e_{x^{yz}}e_t)
+f(e_y  e_t)g(e_{x^{yt}}e_{z^t})
\]
This formula is to be compared with equation (23) of \cite{Cl}. A full explanation of this agreement is given in next section.
\end{example}
\section{An explicit expression for the coproduct}
We give here an explicit formula for $\Delta(e_{x_1}\cdots e_{x_n})$ for any $x_1,\ldots,x_n$ in the rack $X$, thus generalizing Example \ref{exDelta}. For this we have to introduce several notations: for any $n\ge 1$ and for any $i\in\{1,\ldots, n\}$ we define two maps $\delta_i^0$ and $\delta_i^1$ from $X^n$ to $X^{n-1}$ by:
\begin{eqnarray*}
\delta_i^0(x_1,\ldots,x_n)&=&(x_1,\ldots,x_{i-1},x_{i+1},\ldots,x_n),\\
\delta_i^1(x_1,\ldots,x_n)&=&(x_1\t x_i,\ldots,x_{i-1}\t x_i,x_{i+1},\ldots,x_n).
\end{eqnarray*}
The above identification of $B$ with $A\otimes TX$ given by $ae_{x_1}\cdots e_{x_n}\simeq a\otimes(x_1,\ldots,x_n)$ permits to promote $\delta_i^0$ and $\delta_i^1$ to $A$-linear endomorphisms of $B$:
\begin{eqnarray*}
\delta_i^0(ae_{x_1}\cdots e_{x_n})&=&ae_{x_1}\dots e_{x_{i-1}}e_{x_{i+1}}\cdots e_{x_n}\hbox{ if }i\le n,\\
	&=& 0\hbox{ if }i>n,\\
\delta_i^1(ae_{x_1}\cdots e_{x_n})&=&ax_ie_{x_1\t x_i}\dots e_{x_{i-1}\t x_i}e_{x_{i+1}}\cdots e_{x_n}\hbox{ if }i\le n,\\
	&=& 0\hbox{ if }i>n.
\end{eqnarray*}
A straightforward computation using the rack axioms shows:
\begin{equation}\label{cube-sets}
\delta_i^\varepsilon\delta_{j}^\eta=\delta_{j-1}^\eta\delta_i^\varepsilon
\end{equation}
for any $i<j$ and for any $\varepsilon,\eta\in\{0,1\}$. Identities \eqref{cube-sets} are the defining axioms for $\square$-sets \cite[Paragraph 3.1]{Cl}. Note that the boundary \eqref {bord-racks} can be rewritten as:
\begin{equation}
\partial=\sum_{i\ge 1} (-1)^i(\delta_i^0-\delta_i^1).
\end{equation}
For any finite subset $A$ of $\mathbb N_{>0}=\{1,2,3\ldots\}$ and for $\varepsilon\in\{0,1\}$, we denote by $\delta_A^\varepsilon$ the composition of the maps $\delta_a^\varepsilon$ for $a\in A$ displayed in the increasing order.
\begin{prop}\label{coprod-explicit}
\begin{equation}\label{coprod-formula}
\Delta(e_{x_1}\cdots e_{x_n})=\sum_{A\subset\{1,\ldots,n\}}\epsilon(A)\delta_A^0(e_{x_1}\cdots e_{x_n})\otimes \delta_{A^c}^1(e_{x_1}\cdots e_{x_n}),
\end{equation}
where $A^c$ is the complement set of $A$ in $\{1,\ldots,n\}$, and where $\epsilon(A)$ is the signature of the unshuffle permutation of $\{1,\ldots,n\}$ which puts $A$ on the left and $A^c$ on the right\footnote{The signature $\epsilon(A)$ differs from that given in \cite[Paragraph 3.1]{Cl} in order to incorporate the extra sign $(-1)^{km}$ therein.}.
\end{prop}
\begin{proof}
We proceed by induction on $n$, the case $n=1$ being immediate.
\begin{eqnarray*}
\Delta(e_{x_1}\cdots e_{x_n})&=&\Delta(e_{x_1}\cdots e_{x_{n-1}})\Delta(e_{x_n})\\
&=&\Delta(e_{x_1}\cdots e_{x_{n-1}})(e_{x_n}\otimes x_n+1\otimes e_{x_n})\\
&=&\left(\sum_{B\subset\{1,\ldots,n-1\}}\epsilon(B)\delta_B^0(e_{x_1}\cdots e_{x_{n-1}})\otimes \delta_{B^c}^1(e_{x_1}\cdots e_{x_{n-1}})
\right)(e_{x_n}\otimes x_n+1\otimes e_{x_n})\\
&=&\sum_{B\subset\{1,\ldots,n-1\}}(-1)^{\vert B\vert}\epsilon(B)\delta_B^0(e_{x_1}\cdots e_{x_{n-1}})e_{x_n}\otimes \delta_{B^c}^1(e_{x_1}\cdots e_{x_{n-1}})x_n\\
&&+\sum_{B\subset\{1,\ldots,n-1\}}\epsilon(B)\delta_B^0(e_{x_1}\cdots e_{x_{n-1}})\otimes \delta_{B^c}^1(e_{x_1}\cdots e_{x_{n-1}})e_{x_n}\\
&=&\sum_{A\subset\{1,\ldots,n\}}\epsilon(A)\delta_A^0(e_{x_1}\cdots e_{x_n})\otimes \delta_{A^c}^1(e_{x_1}\cdots e_{x_n}).
\end{eqnarray*}
\end{proof}
\begin{cor}\label{clauwens}
The cup-product \eqref{cup} coincides with the cup-product given by F. J.-B. J. Clauwens in \cite[Equation (32)]{Cl}.
\end{cor}
\begin{proof}
Considering our definition \eqref{cup}, this is immediate by comparing \eqref{coprod-formula} with Equation (32) defining the cup-product in \cite{Cl}.
\end{proof}
\section{Graded commutativity of the cup-product}
From Theorem \ref{thm:main} we immediately get an associative graded product, still denoted by $\smile$, on the rack cohomology space $H^\bullet(X)$. The graded commutativity is granted from Corollary \ref{clauwens}: in fact the product defined by Clauwens in \cite{Cl} is (super) commutative, since it is the cohomology of a topological space. We give here a direct algebraic proof based on a homotopy argument. Graded commutativity is illustrated by the following low-degree computation:
\begin{example}
Let $f,g\in C ^1(X,k)$, we identify them with degree one $A$-linear maps from $B$ to $k$
(also denoted  by $f$ and $g$) determined by the values
$f(e_x):=f(x)$ and $g(e_x):=g(x)$ for $x\in X$. We have $(f\smile g)\in C^2(X,k)$,
and it is defined by
\[(f\smile g)(e_xe_y)=
(f\ot g)\Delta(e_xe_y)
=(f\ot g)(
e_xe_y\ot xy+1\ot e_xe_y+
e_x\ot xe_y
-e_y\ot ye_{x^y})
\]
(see example \ref{exDelta} for the comultiplication of $e_xe_y$).
 Since $f$ and $g$ are of degree one, they vanish on $e_xe_y$ and on 
elements of $A$, so the only remaining terms are
\[
-f(e_x)g( xe_y)
+f(e_y)g( ye_{x^y})=
-f(e_x)g( e_y)
+f(e_y)g(e_{x^y})
\]
since $g$ (and $f$) is left $A$-linear and the action of $x$ and $y$ on $k$ is trivial.
One sees with this example that the product is in general  not commutative. In fact, it is commutative if and only if
the rack is trivial: in such a case, the product agrees with the shuffle product on the (graded) dual of
the tensor algebra.
On the other hand, if $f$ and $g$ are 1-cocycles, the condition $\partial ^*f=0$ means
exactly
$f(e_x)=f(e_{x^y})$ for all $x$ and $y$, so we see that this product restricted to 1-cocycles
is commutative.
\end{example}
\begin{lem}\label{com-one}
Let $h:B\to B\ot^D B$ the degree one $A$-module morphism defined as follows:
\begin{eqnarray*}
h(1)&:=&0,\\
h(e_x)&:=&e_x\otimes e_x,\\
h(e_{x_1}\cdots e_{x_n})&:=&\sum_{i=1}^n (-1)^{i-1}(\tau\circ\Delta)(e_{x_1})\cdots (\tau\circ\Delta)(e_{x_{i-1}})h(e_{x_i})\Delta(e_{x_{i+1}})\cdots\Delta(e_{x_n}),
\end{eqnarray*}
where $\tau:B\ot^D B\to B\ot^D B$ is the signed flip of the two factors. Then for any $a,b\in B$ we have:
\begin{equation}\label{eq:com-one}
h(ab)=h(a)\Delta(b)+(-1)^{|a|}(\tau\circ\Delta)(a)h(b).
\end{equation}
\end{lem}
\begin{proof}
It is sufficient to check \eqref{eq:com-one} on $a=e_{x_1}\cdots e_{x_p}$ and $b=e_{x_{p+1}}\cdots e_{x_{p+q}}$. This is immediate form the definition, using the fact that both $\Delta$ and $\tau\circ\Delta$ are algebra morphisms.
\end{proof}
\noindent For example, an easy computation gives:
\begin{equation}\label{exey}
h(e_xe_y)=(xe_y+e_x)\otimes e_xe_y-e_xe_y\otimes(e_xy+e_y).
\end{equation}
\ignore{
\begin{cor}\label{com-two}
For any $a_1,\ldots,a_n\in B$ the following identity holds:
\begin{equation}\label{eq:com-two}
h(a_1\cdots a_n)=\sum_{i=1}^n(\tau\circ\Delta)(a_1)\cdots(\tau\circ\Delta)(a_{i-1})h(a_i)\Delta(a_{i+1})\cdots\Delta(a_n).
\end{equation}
\end{cor}
}
\begin{prop}\label{homotopy}
The map $h$ is a homotopy between $\Delta$ and $\tau\circ\Delta$, i.e. the following holds:
\begin{equation}\label{eq:homotopy}
d\circ h+h\circ d=\tau\circ\Delta-\Delta.
\end{equation}
\end{prop}
\begin{proof}
If $x$ is a degree zero generator of $B$ we have $(dh+hd)(x)=0$, and $\Delta(x)=x\otimes x=(\tau\circ\Delta)(x)$, hence \eqref{eq:homotopy} holds. Now for the degree one generator $e_x$ we have:
\begin{eqnarray*}
(dh+hd)(e_x)&=&d(e_x\otimes e_x)+0\\
&=&(x-1)\otimes e_x-e_x\otimes (x-1)\\
&=&x\otimes e_x-e_x\otimes x-1\otimes e_x+e_x\otimes 1\\
&=&(\tau\circ\Delta-\Delta)(e_x).
\end{eqnarray*}
The proof can then be carried out by induction on the degree, using \eqref{eq:com-one}:
\begin{eqnarray*}
(dh+hd)(ab)&=&d\big(h(a)\Delta(b)+(-1)^{|a|}(\tau\circ \Delta)(a)h(b)\big)+h\big(da.b+(-1)^{|a|}a.db\big)\\
&=&dh(a)\Delta(b)+(-1)^{|a|+1}h(a)d\Delta(b)+(-1)^{|a|}d(\tau\circ\Delta)(a)h(b)+(\tau\circ\Delta)(a)dh(b)\\
&&+hd(a)\Delta(b)+(-1)^{|a|+1}(\tau\circ\Delta)(da)h(b)+(-1)^{|a|}h(a)\Delta(db)+(\tau\circ\Delta)(a)hd(b)\\
&=&(\tau\circ\Delta-\Delta)(a)\Delta(b)+(\tau\circ\Delta)(a)(\tau\circ\Delta-\Delta)(b)\\
&=&(\tau\circ\Delta)(a)(\tau\circ\Delta)(b)-\Delta(a)\Delta(b)\\
&=&(\tau\circ\Delta-\Delta)(ab).
\end{eqnarray*}
\end{proof}
\begin{thm}\label{cup-product}
The map $h$ induces a homotopy between $\smile$ and $\smile^{op}$ in $C^\bullet(X)$, in particular, this gives an algebraic proof of the fact that $H^\bullet(X)$ is graded commutative.
\end{thm}
\begin{proof}
\noindent The cup-product of two cochains $f$ and $g$ is given by the convolution product:
$$f\smile g=\mu\circ (f\ot g)\circ\Delta$$
where $\mu$ is the multiplication in the base field $k$. Hence we have for any homogeneous $x\in B$ of degree $|f|+|g|$:
\begin{eqnarray*}
\big((f\smile g-(-1)^{|f||g|}g\smile f\big)(x)&=&\sum_{(x)}f(x_1)g(x_2)-(-1)^{|f||g|}g(x_1)f(x_2)\\
&=&\sum_{(x)}f(x_1)g(x_2)-(-1)^{|x_1||x_2|}g(x_1)f(x_2)\\
&=&\sum_{(x)}f(x_1)g(x_2)-(-1)^{|x_1||x_2|}f(x_2)g(x_1)\\
&=&\mu\circ (f\otimes g)\circ (\Delta-\tau\Delta)(x)\\
&=&-\mu\circ (f\otimes g)\circ (hd+dh)(x).
\end{eqnarray*}
Hence $H:\mop{Hom}_A(B,k)^{\ot 2}\to\mop{Hom}_A(B,k)$ defined by:
$$H(f\ot g):=-\mu\circ (f\otimes g)\circ h$$
is a homotopy between $\smile$ and $\smile\circ\tau$. A standard argument then gives us $\smile=\smile\circ\tau$ on the cohomology space.
\end{proof}
\section{Quandle cohomology}
\noindent If the rack $X$ is a quandle, then the complex $C(X,k)$ has a degeneration
subcomplex:
\[
 C_n^D(X,k)=\langle (x_1,\dots,x_{i-1},y,y,x_{i+2},\dots,x_n)\rangle
\]
and the appropriate complexes for computing quandle (co)homology  are defined by
\[
 C_\bullet^Q(X,k):=C_\bullet(X,k)/C_\bullet^D(X,k)\hbox{, and }
 C^\bullet_Q(X,k):=\Hom(C^Q_\bullet(X,k),k)
\]
Similarly to Lemma \ref{Bvscomplejo}, one can easily see the following:
\begin{lem}
Define the graded algebra $B^Q:=B/\langle e_x^2: x\in X\rangle$, then $B^Q$ inherits
the differential and there is an isomorphism of complexes
\[
 (C^Q_\bullet,\partial)\cong (k\ot_AB^Q_\bullet,\Id_k\ot_Ad)\hbox{, and }(C^\bullet,\partial^*)\cong (\Hom_{A-}(B^Q_\bullet,k),d^*)
\]
\end{lem}
\ignore{
The problem is that the ideal $\langle e_x^2:x\in X\rangle$ is {\em not} a coideal, since
\[
 \Delta(e_x^2)
=\Delta(e_x)\Delta(e_x)
=
(e_x\ot x+1\ot e_x)(e_x\ot x+1\ot e_x)
=
e_x^2\ot x^2
+e_x\ot x e_x
+e_x\ot e_x x
+1\ot e_x^2
\]
In this context, it is convenient to introduce signs, and consider the super-version of $B$. Namely, since
$B$ is graded, one can consider the graded tensor product on $B\ot B$ defined by
\[
 (a\ot b)(a'\ot b')=(-1)^{|a'||b|}aa'\ot bb'
\]
if $a'$ and $b$ are homogeneous of degree
$|a'|$ and $|b|$ respectively. Notice that the only part of the algebraic structure on 
$B$ afected by this change is $\Delta$, and now we consider $\wt\Delta:B\to B\ot B$
in a similar fashion as before, namely
\[
 \wt\Delta(x)= \Delta(x)=x\ot x
\]
\[
 \wt\Delta(e_x)= \Delta(e_x)=e_x\ot x+1\ot e_x
\]
but we extend $\wt\Delta$ multiplicatively using the signed-product for $B\ot B$. All results on the previous section 
remain valid with apropriated introduction of signs. For example,
\[\wt\Delta(e_xe_y)=e_xe_y\ot xy+1\ot e_xe_y+
e_x\ot xe_y
-e_y\ot ye_{x^y}
\]
\[
\wt\Delta(e_xe_ye_ze_t)=
e_xe_y\ot xye_ze_t
+e_ze_t\ot zx e_x^{zt}e_y  ^{zt}
-e_x e_z\ot   xze_{y^z}e_t\]
\[
-e_xe_t   \ot xte_{y^t}e_{z^t}
+e_y e_z\ot yze_{x^{yz}}e_t
-e_y  e_t\ot yte_{x^{yt}}e_{z^t}+\cdots
\]
and for $f$ and $g$ in $C^2$, the formula for
$(f\smile g)(e_xe_ye_ze_t)$ modifies into
\[
f(e_xe_y)g(e_ze_t)+f(e_ze_t)g( e_x^{zt}e_y  ^{zt})
-f(e_x e_z)g(e_{y^z}e_t)\]
\[
-f(e_xe_t)g(e_{y^t}e_{z^t})
+f(e_y e_z)g(e_{x^{yz}}e_t)
-f(e_y  e_t)g(e_{x^{yt}}e_{z^t})
\]
}
\noindent The advantage of the signed-comultiplication is that 
\[\wt\Delta(e_x^2)
=e_x^2\ot x^2+1\ot e_x^2+
e_x\ot xe_x
-e_x\ot xe_{x^x}
\]
\[=e_x^2\ot x^2+1\ot e_x^2
\]
and so, $\langle e_x^2:x\in X\rangle$ is a coideal in $(B,\wt\Delta)$, hence,
$B^Q$ inherits a comultiplication compatible with the differential, and
$C^Q_\bullet(X,k)$ is a differential {\em algebra}. From Equation \eqref{exey} with $x=y$, we immediately get that the homotopy $h$ gives rise to a homotopy $h^Q:B^Q\to B^Q\ot^D B^Q$. Hence the quandle cohomology is also a graded commutative algebra.
\section{Coefficients}
Given a rack $X$ and a set $Y$, one says that $Y$ has the structure of an $X$-set if there is given
a map
$
*:Y\times X\to Y       
 $
verifying:
\begin{itemize}
 \item $-*x:Y\to Y$ is a bijection for all $x\in X$,
\item $(y*x)*x'=(y*x')*(x\t x')$ for all $y\in Y$, $x,x'\in X$.
\end{itemize}

A first example is $Y=X$ with $*=\t$. A second example is $Y=\{1\}$ with $1*x=1$ for all $x\in X$. Given
$Y$ an $X$-set, then it is clear that the axioms of $X$ sets imply that $k[Y]$, the $k$-free module on $Y$,
is a right $A(X)$-module (for instance, $k\{1\}\cong k$). In fact, since multiplication by $x$ is a bijection on $Y$, one see that
$k[Y]$ is actually a module over $k[G_X]$.
 Notices that, for  $k[G_X]$, right modules are in correspondence with left modules. 
In \cite{Cl}, the author consider $X$-sets as ``coefficients''
for rack/quandle (co)homology.  In our setting, it is for free to consider right $A$-modules,
and define
\[
 C_\bullet(X,M)=M\ot_A B_\bullet
\]
and, for left  $A$-modules
\[
 C^\bullet(X,M)=\Hom_{A}(B_\bullet,M)
\]
Given $N$ and $M$ two left $A$-modules, the comultiplication in $B$, together with the diagonal
$A$-structure on $N\ot M$, denoted $M\ot^DN$, gives a map

\[
\xymatrix{
\ar@/_/[rrd]C^\bullet(X,M)\ot C^\bullet(X,M)\ar@{=}[r]&\Hom_{A-}(B,N)\ot \Hom_{A-}(B,M)\ar[r]^{i|}&
\Hom_{A-}(B\ot^DB,N\ot^DM)\ar[d]^{\Delta^*}\\
&&\Hom_{A-}(B,N\ot ^D\!M)
}\]
which is associative in an obvious way. Similarly  for $X$  a quandle, replacing $B$ by $B^Q$.

\end{document}